\documentclass[11pt,a4paper]{amsart}
\usepackage{amssymb,amsmath}
\usepackage[dvips]{color}
\usepackage[dvips,final]{graphics}
\usepackage{epstopdf}
\usepackage{amscd}

\usepackage{pdfsync}

\theoremstyle{plain}
\newtheorem{theorem}{Theorem}[section]

\newtheorem{lemma}[theorem]{Lemma}
\theoremstyle{definition}
\newtheorem{definition}[theorem]{Definition}

\theoremstyle{remark}
\newtheorem{remark}[theorem]{Remark}

\begin{document}

\bibliographystyle{plain}

\newcommand{\Rn}{\mathbb R^n}
\newcommand{\E}{\mathbb E}
\newcommand{\Rm}{\mathbb R^m}
\newcommand{\rn}[1]{{\mathbb R}^{#1}}
\newcommand{\R}{\mathbb R}
\newcommand{\C}{\mathbb C}
\newcommand{\G}{\mathbb G}
\newcommand{\M}{\mathbb M}
\newcommand{\Z}{\mathbb Z}
\newcommand{\D}[1]{\mathcal D^{#1}}
\newcommand{\Ci}[1]{\mathcal C^{#1}}
\newcommand{\Ch}[1]{\mathcal C_{\mathbb H}^{#1}}
\renewcommand{\L}[1]{\mathcal L^{#1}}
\newcommand{\BVG}{BV_{\G}(\Omega)}
\newcommand{\supp}{\mathrm{supp}\;}

\newcommand{\dom}{\mathrm{Dom}\;}
\newcommand{\Ast}{\un{\ast}}

\newcommand{\average}{{\mathchoice {\kern1ex\vcenter{\hrule height.4pt
width 6pt depth0pt} \kern-9.7pt} {\kern1ex\vcenter{\hrule height.4pt width
4.3pt depth0pt}
\kern-7pt} {} {} }}
\newcommand{\ave}{\average\int}

\newcommand{\hhd}[1]{{\mathcal H}_d^{#1}}
\newcommand{\hsd}[1]{{\mathcal S}_d^{#1}}

\newcommand{\he}[1]{{\mathbb H}^{#1}}
\newcommand{\hhe}[1]{{H\mathbb H}^{#1}}

\newcommand{\cov}[1]{{\bigwedge\nolimits^{#1}{\mfrak g}}}
\newcommand{\vet}[1]{{\bigwedge\nolimits_{#1}{\mfrak g}}}

\newcommand{\covw}[2]{{\bigwedge\nolimits^{#1,#2}{\mfrak g}}}

\newcommand{\vetfiber}[2]{{\bigwedge\nolimits_{#1,#2}{\mfrak g}}}
\newcommand{\covfiber}[2]{{\bigwedge\nolimits^{#1}_{#2}{\mfrak g}}}

\newcommand{\covwfiber}[3]{{\bigwedge\nolimits^{#1,#2}_{#3}{\mfrak g}}}

\newcommand{\covv}[2]{{\bigwedge\nolimits^{#1}{#2}}}
\newcommand{\vett}[2]{{\bigwedge\nolimits_{#1}{#2}}}

\newcommand{\covvfiber}[3]{{\bigwedge\nolimits^{#1}_{#2}{#3}}}

\newcommand{\vettfiber}[3]{{\bigwedge\nolimits_{#1,#2}{#3}}}

\newcommand{\covn}[2]{{\bigwedge\nolimits^{#1}\rn {#2}}}
\newcommand{\vetn}[2]{{\bigwedge\nolimits_{#1}\rn {#2}}}
\newcommand{\covh}[1]{{\bigwedge\nolimits^{#1}{\mfrak h_1}}}
\newcommand{\veth}[1]{{\bigwedge\nolimits_{#1}{\mfrak h_1}}}
\newcommand{\hcov}[1]{{_H\!\!\bigwedge\nolimits^{#1}}}
\newcommand{\hvet}[1]{{_H\!\!\bigwedge\nolimits_{#1}}}

\newcommand{\covf}[2]{{\bigwedge\nolimits^{#1}_{#2}{\mfrak h}}}
\newcommand{\vetf}[2]{{\bigwedge\nolimits_{#1,#2}{\mfrak h}}}
\newcommand{\covhf}[2]{{\bigwedge\nolimits^{#1}_{#2}{\mfrak h_1}}}
\newcommand{\vethf}[2]{{\bigwedge\nolimits_{#1,#2}{\mfrak h_1}}}
\newcommand{\hcovf}[2]{{_H\!\!\bigwedge\nolimits^{#1}_{#2}}}
\newcommand{\hvetf}[2]{{_H\!\!\bigwedge\nolimits_{#1,#2}}}

\newcommand{\defin}{\stackrel{\mathrm{def}}{=}}
\newcommand{\gradh}{\nabla_H}
\newcommand{\current}[1]{\left[\!\left[{#1}\right]\!\right]}
\newcommand{\scal}[2]{\langle {#1} , {#2}\rangle}
\newcommand{\escal}[2]{\langle {#1} , {#2}\rangle_{\mathrm{Euc}}}
\newcommand{\Scal}[2]{\langle {#1} \vert {#2}\rangle}
\newcommand{\scalp}[3]{\langle {#1} , {#2}\rangle_{#3}}
\newcommand{\dc}[2]{d_c\left( {#1} , {#2}\right)}
\newcommand{\res}{\mathop{\hbox{\vrule height 7pt width .5pt depth 0pt
\vrule height .5pt width 6pt depth 0pt}}\nolimits}
\newcommand{\norm}[1]{\left\Vert{#1}\right\Vert}
\newcommand{\modul}[2]{{\left\vert{#1}\right\vert}_{#2}}
\newcommand{\perh}{\partial_\mathbb H}

\newcommand{\ccheck}{{\vphantom i}^{\mathrm v}\!\,}

\newcommand{\wcheck}{{\vphantom i}^{\mathrm w}\!\,}

\newcommand{\mc}{\mathcal }
\newcommand{\mbf}{\mathbf}
\newcommand{\mfrak}{\mathfrak}
\newcommand{\mrm}{\mathrm}
\newcommand{\no}{\noindent}
\newcommand{\dis}{\displaystyle}

\newcommand{\U}{\mathcal U}
\newcommand{\ga}{\alpha}
\newcommand{\gb}{\beta}
\newcommand{\gga}{\gamma}
\newcommand{\gd}{\delta}
\newcommand{\eps}{\varepsilon}
\newcommand{\gf}{\varphi}
\newcommand{\GF}{\varphi}
\newcommand{\gl}{\lambda}
\newcommand{\GL}{\Lambda}
\newcommand{\gp}{\psi}
\newcommand{\GP}{\Psi}
\newcommand{\gr}{\varrho}
\newcommand{\go}{\omega}
\newcommand{\gs}{\sigma}
\newcommand{\gt}{\theta}
\newcommand{\gx}{\xi}
\newcommand{\GO}{\Omega}

\newcommand{\Wedge}{\buildrel\circ\over \wedge}

\newcommand{\WO}[4]{\mathop{W}\limits^\circ{}\!_{#4}^{{#1},{#2}}
(#3)}

\newcommand{\GH}{H\G}
\newcommand{\N}{\mathbb N}

%

\newcommand{\Nhmin}{N_h^{\mathrm{min}}}
\newcommand{\Nhmax}{N_h^{\mathrm{max}}}

\newcommand{\Mhmin}{M_h^{\mathrm{min}}}
\newcommand{\Mhmax}{M_h^{\mathrm{max}}}

\newcommand{\un}[1]{\underline{#1}}

\newcommand{\curl}{\mathrm{curl}\;}
\newcommand{\curlh}{\mathrm{curl}_{\he{}}\;}
\newcommand{\hd}{\hat{d_c}}
\newcommand{\divg}{\mathrm{div}_\G\,}
\newcommand{\divgh}{\mathrm{div}_{\hat\G}\,}
\newcommand{\divh}{\mathrm{div}_{\he{}}\,}
\newcommand{\e}{\mathrm{Euc}}


\newcommand{\hatcov}[1]{{\bigwedge\nolimits^{#1}{\hat{\mfrak g}}}}

\title[
] { An application of the Theorem on Sums to viscosity solutions of degenerate fully nonlinear equations
}

\author{Fausto Ferrari}
\address{Dipartimento di Matematica dell'Universit\`a di Bologna, Piazza di Porta S. Donato, 5, 40126 Bologna, Italy.}
\email{\tt fausto.ferrari@unibo.it}

\thanks{The author is supported by  MURST, Italy, and
INDAM-GNAMPA project 2017: {\it Regolarit\`a delle soluzioni viscose per equazioni a derivate parziali non lineari degeneri}}

\thanks{The author wishes to thank P. Mannucci for pointing out her joint papers \cite{MMT1} and \cite{MMT2}.}

\keywords{Degenerate elliptic operators, Nonlinear elliptic operators, Carnot groups, viscosity solutions, Theorem on sums, H\"older regularity}

\subjclass{35D40, 35B65, 35H20.
}

\begin{abstract} 
 We prove  H\"older  continuous regularity of bounded, uniformly continuous,  viscosity solutions  of  degenerate fully nonlinear equations defined in all of $\mathbb{R}^n$ space. In particular the result applies also to some operators in Carnot groups.
\end{abstract}

\maketitle

\tableofcontents

\section{Introduction}

In this paper we continue the study of the properties of the viscosity solutions of some nonlinear PDEs started in \cite{F} and \cite{FeVe}.  In those papers we faced the case of 
 non-divergence nonlinear equations modeled on vector fields
in the Heisenberg group. We  proved there that bounded uniformly continuous functions that are  also viscosity solutions of some nonlinear degenerate uniformly elliptic equations in all the Heisenberg group $\mathbb{H}^1$   are also H\"older continuous in the classical sense.

In those papers we did not need to prove Harnack inequality in advance, as it is customary to do in order to obtain that type of result.

Here we deal with a larger classes of operators, intrinsically uniformly elliptic even if the operators are not elliptic in the classical sense defined in \cite{CC}, obtaining similar regularity results to the ones  proved in \cite{F} and \cite{FeVe}, that is without proving first a Harnack inequality. 
 
 Our research has been originated in reading \cite{Ishii1}. In that paper the author applied the Theorem on sums, see \cite{Crandall}, to an elliptic linear operator having quite smooth coefficients. 
 
 The key point that we exploit in our approach is based on the existence of square root matrices, sufficiently smooth, of the symmetric matrix associated with the second order term of the equation, the so called leading term. 
 
Since in this paper we consider many different families of operators in non-divergence form, we prove that our approach works even in those cases in which, instead of the classical square root matrix, there exist  rectangular square roots matrices $\sigma$ such that $\sigma^T\sigma$ becomes, possibly, the degenerate square matrix describing the second order term of the equation. A typical application of this representation appears in Carnot groups, but many other examples exist.

In order to better explain the result, we introduce the classes of the operators that we deal with. In the sequel we denote with $S^m$ the set of $m\times m$ square symmetric matrices, $m\in \mathbb{N},$ $n\geq 1.$ 

\begin{definition} \label{def1gen}

Let $0<\lambda\leq\Lambda$ be given real  numbers and let $0<m\leq n$ be two positive integers. 
Let  $\sigma$ be  a $m\times n$ matrix, $m\leq n,$ with Lipschitz continuous coefficients defined in $\Omega\subseteq \mathbb{R}^n.$
Let $G:S^m\to \mathbb{R}$ be a given function such that  for every $A,B\in S^m,$ if $B\leq A$ then 
$$
\lambda\mbox{Tr}(A-B)\leq G(A)-G(B)\leq \Lambda \mbox{Tr}(A-B).
$$

We define the function $F:S^n\times\Omega\to \mathbb{R}$ in such a way that,  for every $M\in S^n$ and for every $x\in\Omega$
$$
F(M,x)=G( \sigma(x)M\sigma(x)^T).
$$
We sometime denote for every $x\in \Omega,$ the $n\times n$ matrix $P(x)=\sigma(x)^T\sigma(x).$  
\end{definition}

We postpone to the Section \ref{examples_tools} some comments about the novelty of this family of operators and we state immediately our main result.

\begin{theorem}\label{mainregularityresult}
 Let  $f,c\in C(\mathbb{R}^n)$ be continuous functions and let $L_c,$ $L_f,$ $\beta,$ $\beta'$  be positive real numbers such that  $\beta,\beta'\in (0,1]$ and $\forall$ $x,y\in \mathbb{R}^n,$ $
 |c(x)-c(y)|\leq L_c|x-y|^{\beta}$, $|f(x)-f(y)|\leq L_c|x-y|^{\beta'}.$ Let us suppose $\inf_{x\in  \mathbb{R}^n}c(x):=c_0>0.$ Let $F$ be an operator satisfying Definition \ref{def1gen} where $\sigma$ is Lipschitz continuous in $\mathbb{R}^n$ and $P=\sigma^T\sigma.$ 
 Assume that there exists a positive constant $\bar{c},$  $c_0\geq\bar{c}>0.$ 
 If $u\in C(\mathbb{R}^n)$ is a bounded uniformly continuous viscosity solution of the equation
 $$F(D^{2}u(x),x)-c(x)u(x)= f(x),\quad \mathbb{R}^n,$$   
 and
 \begin{equation}\label{123ip}
\limsup_{|x|\to\infty}\left(\frac{\mbox{Tr}(P(x))}{|x|^2}-\frac{c_0}{2\Lambda}\right)\leq 0,
\end{equation}
 then there exist $0<\alpha:=\alpha(c_0,\bar{c}, L_c,L_f,\Lambda)\in (0,1],$ $\alpha\leq \min\{\beta,\beta'\},$ and $L:=L(c_0,\bar{c}, L_c,L_f,\Lambda)>0$ such that for every $x,y\in\mathbb{R}^n$
 $$
 |u(x)-u(y)|\leq L|x-y|^\alpha.
 $$ 
\end{theorem}

We point out that in our presentation we do not distinguish the operators by considering their possible degenerateness, since the approach that we introduce applies independently to the fact that the operator is degenerate elliptic or it is not. 

In fact it is well known that viscosity theory existence  is independent to the lack of ellipticity. Namely the construction of Perron solutions can be done independently to the ellipticity of the equation. 

As a consequence, even when we deal with PDEs in Carnot groups, we state our results always considering regularity properties with respect to the classical notions of regularity. For instance, in our main result we obtain H\"older regularity of the viscosity solutions in the classical sense. 

We point out this aspect since, on the other hand, there exists also a literature that deal with intrinsic regularity results, see e.g. \cite{LiuManfrediZhou}. 
In particular those results are stated using intrinsic notions associated with the geometry of the operator considered. From this point of view, we recall that the intrinsic distance associated with degenerate PDEs, usually, is not equivalent to the Euclidean one. For the reader convenience we shall come back at the end of Section \ref{examples_tools} on this remark. 

After this introduction, the paper is organized as follows: in Section \ref{examples_tools} we list some cases to which our result applies and we introduce the main tools we need to; in Section  \ref{mainregularityresulttit} we show the proof of Theorem  \ref{mainregularityresult} 
and in Section \ref{concorem} we discuss some final remarks and conclusions. Concerning the recent literature about this subject, in addition to  \cite{F} and \cite{FeVe}, we like to cite also, \cite{MMT1} and \cite{MMT2}. 
\section{Examples and preliminary tools}\label{examples_tools}
We begin this section by listing some examples of the operators belonging to the family introduced in Definition \ref{def1gen}. All the fully nonlinear operators $F,$ that are uniformly elliptic, see \cite{CC}, belong to our class when $P\equiv I.$ In this case $\sigma=I\in S^n$ and $m=n.$

In addition, in order to give an explicit nontrivial example belonging to the class of  fully nonlinear operator studied in \cite{CC}, we consider in $\mathbb{R}^3$ the matrix 
 \begin{equation*}
P_{\mathbb{H}^1}(x)=\left[
\begin{array}{llc}
1,&0,&2x_2\\
0,&1,&-2x_1\\
2x_2,&-2x_1,&4(x_1^2+x_2^2)
\end{array}
\right],
\end{equation*}
 
so that

\begin{equation*}
\sqrt{P_{\mathbb{H}^1}(x)}
 =\left[
 \begin{array}{lll}
\frac{ x_2^2+\frac{x_1^2}{\sqrt{1+4(x_1^2+x_2^2)}}}{x_1^2+x_2^2},&\frac{x_1x_2(1-\frac{1}{\sqrt{1+4(x_1^2+x_2^2)}})}{x_1^2+x_2^2},&\frac{2x_2}{\sqrt{1+4(x_1^2+x_2^2)}}\\
\frac{x_1x_2(1-\frac{1}{\sqrt{1+4(x_1^2+x_2^2)}})}{x_1^2+x_2^2},&\frac{ x_1^2+\frac{x_2^2}{\sqrt{1+4(x_1^2+x_2^2)}}}{x_1^2+x_2^2},&-\frac{2x_1}{\sqrt{1+4(x_1^2+x_2^2)}}\\
\frac{2x_2}{\sqrt{1+4(x_1^2+x_2^2)}},&-\frac{2x_1}{\sqrt{1+4(x_1^2+x_2^2)}},&\frac{4(x_1^2+x_2^2)}{\sqrt{1+4(x_1^2+x_2^2)}}
 \end{array}
 \right].
 \end{equation*}

In the class of our operators we find the following ones

\begin{equation*}\begin{split}
&\mathcal{P}^+_{\mathbb{H}^1}(M,x)
=
\max_{A\in\mathcal{A}_{\lambda,\Lambda}}\mbox{Tr}(A\sqrt{P(x)}M\sqrt{P(x)})
\end{split}
\end{equation*}
and
\begin{equation*}
\begin{split}
&\mathcal{P}^-_{\mathbb{H}^1}(M,x)
=\min_{A\in\mathcal{A}_{\lambda,\Lambda}}\mbox{Tr}(A\sqrt{P(x)}M\sqrt{P(x)}),
\end{split}
\end{equation*}
where 
$$
\mathcal{A}_{\lambda,\Lambda}=\{A\in S^3:\quad \lambda |\xi|^2\leq \langle A\xi,\xi\rangle\leq\Lambda |\xi|^2\}.
$$

They are the analogous ones of the Pucci's extremal operators belonging to the class of fully nonlinear uniformly elliptic operators, see \cite{CC}.

In this framework, also the particular case given by the sublaplacian in the Heisenberg group
$$
\Delta_{\mathbb{H}^1}u(x)\equiv\mbox{Tr}(P_{\mathbb{H}^1}(x)D^2u(x))=G(D^2u(x))=F(D^2u(x),x),
$$
where 
$$
G(M)=\mbox{Tr}(\sqrt{P_{\mathbb{H}^1}(x)}M\sqrt{P_{\mathbb{H}^1}(x)}),
$$
belong to the same class. 

Indeed
$$
\lambda \Delta_{\mathbb{H}^1}u(x)\leq G(\sigma_{\mathbb{H}^1} (x) D^2u(x) \sigma(x)^T_{\mathbb{H}^1})\leq \Lambda \Delta_{\mathbb{H}^1}u.
$$

So that we conclude that this operators are not uniformly elliptic in the classical sense described in \cite{CC}.

It is worth to say that we can also consider those operators $F$ obtained by our definition remarking that if $\sigma$ is not a squared matrix, but simply a rectangular matrix, we can construct, at least apparently, another family of operators.

For example, one more time considering for simplicity the Heisenberg group $\mathbb{H}^1$,  that is the simplest case of a nontrivial Carnot group, we have:

$P_{\mathbb{H}^1}(x)=\sigma^T(x)\sigma(x)$  where:
 \begin{equation*}
\sigma_{\mathbb{H}^1}(x)=\left[
\begin{array}{llc}
1,&0,2x_2\\
0,&1,-2x_1
\end{array}
\right].
\end{equation*}
As a consequence for every $M\in S^{3\times 3}$
$$
F(M,x)=G(\sigma_{\mathbb{H}^1}(x) M \sigma_{\mathbb{H}^1}(x)^T).
$$
This approach can be extended to every Carnot group considering the matrix $\sigma_{\mathbb{G}}$ given by the coefficient that determine the vector fields of the first stratum of the Lie algebra in a Carnot group $\mathbb{G}$, namely we construct the matrix
$$
\sigma^T_{\mathbb{G}}(x)=\left[\begin{array}{l}X_1(x)\\
X_2(x)\\
\vdots\\
X_m(x)
\end{array}
\right],
$$
where
$$
\mathfrak{g}_1=\mbox{span}\left\{X_1,\dots,X_m\right\},
$$
$$
\mathfrak{g}_2=[\mathfrak{g}_1,\mathfrak{g}_1],\quad \mathfrak{g}_{k+1}=[\mathfrak{g}_1,\mathfrak{g}_{k}],\:\:k\leq p-1,\quad \bigoplus_{j=1}^{p}\mathfrak{g}_j=\mathfrak{g}, 
$$
and $\mathfrak{g}$ is the Lie algebra of the group $\mathbb{G}$ and $p$ is its step. We refer to \cite{BLU} for further details.

It is important to point out that, considering different Carnot groups to the Heisenberg one, our definition 
$$
\lambda \mbox{Tr}(\sigma (x) M \sigma(x)^T)\leq F(M,x)=G(\sigma (x) M \sigma(x)^T)\leq \Lambda \mbox{Tr}(\sigma (x) M \sigma(x)^T)
$$
does not necessary translate into the following equivalent condition
 $$
\lambda \Delta_{\mathbb{G}}u(x)\leq F(D^2u(x),x)=G(\sigma (x) D^2u(x) \sigma(x)^T)\leq  \Lambda \Delta_{\mathbb{G}}u(x),
$$
as in the Heisenberg group. Indeed, it is well known that there exist Carnot groups such that
$$
 \mbox{Tr}(\sigma_{\mathbb{G}} (x) M \sigma(x)^T_{\mathbb{G}})\not=\Delta_{\mathbb{G}}u(x),
$$
where, by definition, $\Delta_{\mathbb{G}}u(x):=\sum_{j=1}^mX_j^2u(x).$

For instance considering the Engel group $\mathbb{E}^1\equiv\mathbb{R}^4$, endowed by the non-commutative inner law
\begin{align}\label{legEngel}
	x\cdot y=&\Big(x_1+y_1,x_2+y_2,x_3+y_3-y_1x_2,x_4+y_4+\frac{1}{2}y_1^2x_2-y_1x_3\Big)
\end{align}
where the Jacobian basis, see \cite{BLU}, is
$$
\begin{matrix}
	X_1=\partial_1-x_2\partial_3-x_3\partial_4 &
	X_2=\partial_2\\
	X_3=\partial_3&
	X_4=\partial_4,
\end{matrix}
$$
the matrix $\sigma_{\mathbb{E}^1}$ becomes
\begin{equation*}
\sigma_{\mathbb{E}^1}(x)=\left[
\begin{array}{llcc}
1,&0,&-x_2,&-x_3\\
0,&1,&0,&0
\end{array}
\right]
\end{equation*}
and
$$
\mbox{Tr}(\sigma_{\mathbb{E}^1} (x) D^2u(x) \sigma(x)^T_{\mathbb{E}^1})=X_1^2u+X_2^2u-x_2\frac{\partial u}{\partial x_4}.
$$
In this case the class of operators that we have defined does not contain explicitly the intrinsic sublaplacian on the Engel group given by $\Delta_{\mathbb{E}^1}u=X_1^2u+X_2^2u.$ Nevertheless $\mbox{Tr}(\sigma_{\mathbb{E}^1} (x) D^2u(x) \sigma(x)^T_{\mathbb{E}^1})$ is still a degenerate operator, having the smallest eigenvalue always $0$ in all of $\mathbb{R}^4,$ see Lemma \ref{stilldegenerate} in the next subsection.

In Carnot groups it is defined a natural distance associated with the geometry of the group called in literature the Carnot-Charath\'eodoty distance. This distance can be constructed in many ways. For instance
if $\mathfrak{g}_1(P)=\mbox{span}\{(P),\dots,X_m(P)\},$ for every $P\in \mathbb{G}$ and  the set
$\{(P),\dots,X_m(P)\}$ is braking generating all the the space $\mathbb{R}^n\equiv\mathbb{G},$ then for every function
$\phi:[0,1]\to \mathbb{G}\equiv\mathbb{R}^n$ parametrizing a path $\gamma\subset\mathbb{G}$ such that for every $t\in[0,1],$ $\phi'(t)\in\mathfrak{g}_1(\phi(t)),$ we define 
$$
\mbox{lenght}(\gamma)=\int_0^1\sqrt{\sum_{k=1}^m\langle \phi'(t),X(\phi(t))\rangle^2}dt,
$$
where $(\gamma,\phi)$ is the horizontal path parametrized by $\phi.$ Then for every $P_0,P_1\in \mathbb{G}$ we define:
\begin{equation*}
\begin{split}
&d_{CC}^{\mathbb{G}}(P_0,P_1)\\
&=\inf\{\mbox{lenght}(\gamma_{P_0,P_1}):\:\: \gamma_{P_0,P_1},\:\:\mbox{is horizontal path connecting $P_0$, with $P_1$} \}
\end{split}
\end{equation*}
as the Carnot-Charath\'eodory distance between $P_0$, with $P_1.$ 
This distance is not equivalent to the Euclidean distance, since it holds only that if $K\subset\mathbb{G},$ is bounded, then there exist $C_1, C_2>0$ such that, for every $P_1,P_2\in K$
$$
C_1|P_1-P_2|_E\leq d_{CC}(P_1,P_2)\leq C_2|P_1-P_2|_E^{\frac{1}{p}},
$$
where $p$ denotes the step of the Carnot group. For instance, in the Heisenberg group $p=2,$ in the Engel group $p=4.$ Thus, as we pointed out in the Introduction, we remark that in the statement of Theorem \ref{mainregularityresult} we make use only of the usual Euclidean distance and the classical H\"older modulus of continuity of the viscosity solutions. Thus all the definitions are given in the classical usual sense.

\subsection{Preliminary tools}

In this subsection we list some useful key tools concerning the eigenvalues of matrices obtained as the product of rectangular matrices and the Theorem of the sums, see \cite{Crandall}.  For the notation and the definition of viscosity solution and other symbols like  sub/super jets $J^{2,\pm}u(x)$, we refer one more time to  \cite{Crandall}.
\begin{lemma}
Let $A$ be a symmetric $n\times n$ matrix such that for every $i=1,\dots,n,$ $a_{ii}>0$ then all the eigenvalues of $A$ are strictly positive. 
\end{lemma}
\begin{lemma}\label{stilldegenerate}
Let $\sigma$ be a $m\times n$ matrix $m\leq n$ such that $\mbox{rank}(\sigma)=m$ then $\sigma\sigma^T$ is an $m\times m$ strictly positive matrix while if $m<n,$ then $\sigma^T\sigma$ is a degenerate matrix whose eigenvalues different to $0$ are the same of
 $\sigma\sigma^T$ and if $m=n$ then $\sigma^T\sigma$  is invertible and its eigenvalues are the same of $\sigma\sigma^T.$
\end{lemma}
\begin{proof}
Let $\lambda$ be an eigenvalue of $\sigma\sigma^T$ and $v$ one of its eigenvectors. Then
 $$\sigma\sigma^Tv=\lambda v,$$
 so that $\langle\sigma\sigma^Tv,v\rangle=\lambda ||v||^2,$ so that $\langle\sigma^Tv,\sigma^Tv\rangle=\lambda ||v||^2,$ that implies that $\lambda>0$ whenever $v\not \in \mbox{Ker}\sigma^T.$ Indeed $v\not \in \mbox{Ker}\sigma^T$   because by hypothesys $\mbox{rank}(\sigma)=m.$ Thus, we conclude that $\sigma\sigma^T$ is an $m\times m$ strictly positive, in particular also invertible, matrix. Consider now an eigenvalue $\lambda$ of the matrix $\sigma\sigma^T.$ If $\lambda\not=0$ and
  $v\in \mbox{Ker}(\sigma\sigma^T-\lambda I)$ then
$$\sigma\sigma^Tv=\lambda v.$$
Thus $\sigma^T\sigma(\sigma^Tv)=\lambda \sigma^Tv,$ that is $\lambda$ is also an eigenvalue of $\sigma^T\sigma.$ This proves that all the nonzero eigenvalues  of $\sigma\sigma^T$ are eigenvalues of $\sigma^T\sigma.$ On the other hand
if $\gamma>0$ is an eigenvalue of $\sigma^T\sigma$ then 
$$\sigma^T\sigma w=\gamma w,$$
$w\in \mbox{Ker}(\sigma^T\sigma-\gamma I),$ then
$$(\sigma\sigma^T)\sigma w=\gamma (\sigma w),$$
then $\gamma$ is also an eigenvector of $\sigma\sigma^T,$ because $\sigma w\not= (0)$ since $\mbox{rank}(\sigma)=m.$ As a consequence the nonzero eigenvalues of $\sigma\sigma^T$ are only the strictly positive eigenvalues of $\sigma^T\sigma.$
The case $m=n$ is now trivial.
\end{proof}
The following result is an obvious consequence of the definition of trace of a matrix.
 \begin{lemma}
  Let $A,B\in S^{n}$ be given. For every $m\times n$ matrices $\sigma_1,\sigma_2$ then
  $$
  \mbox{Tr}(\sigma^T_1\sigma_1A-\sigma^T_2\sigma_2B)= \mbox{Tr}(\sigma_1A\sigma^T_1-\sigma_2B\sigma^T_2).
  $$
  \end{lemma}
  We recall now the maximum principle for semiconvex functions, sometime also named Theorem on the sum, see \cite{Crandall}.
\begin{theorem}[Crandall-Ishii-Lions]\label{CraIsLi}
  Let $\Omega\subseteq\mathbb{R}^n$ be an open set and $u\in USC(\bar{\Omega})$ and $v\in LSC(\bar{\Omega}).$ Let $\phi\in C^{2}(W)$ where $W$ is open and $\Omega\times\Omega \subset W\subseteq\mathbb{R}^n\times\mathbb{R}^n.$
   If there exists $(\hat{x},\hat{y})\in\Omega\time\Omega$ such that
  \begin{equation}
  \begin{split}
  u(\hat{x})-v(\hat{y})-\phi(\hat{x},\hat{y})=\max_{(x,y)\in\overline{\Omega}\times\overline{\Omega}}\left(u(x)-v(y)-\phi(x,y)\right),
  \end{split}
  \end{equation}
  then for each $\mu>0,$ there exist $A=A(\mu)$ and $B=B(\mu)$ such that
  $$
  (D_x\phi(\hat{x},\hat{y}),A)\in\overline{J}^{2,+}u(\hat{x}),\quad (-D_y\phi(\hat{x},\hat{y}),B)\in\overline{J}^{2,-}u(\hat{y}),\quad\mbox{and}
  $$
  \begin{equation*}
  \begin{split}
  -\left(\mu+||D^2\phi(\hat{x},\hat{y})||\right)&
  \left[\begin{array}{cc}
  I,&0\\
  0,&I
  \end{array}
  \right]
 \leq
  \left[\begin{array}{cc}
  A,&0\\
  0,&-B
  \end{array}
  \right]\\
 & \leq D^2\phi(\hat{x},\hat{y})+\frac{1}{\mu}(D^2\phi(\hat{x},\hat{y}))^2.
  \end{split}
  \end{equation*}
  
  Where:
  $$
  D^2\phi(\hat{x},\hat{y})= \left[\begin{array}{cc}
  D^2_{xx}\phi(\hat{x},\hat{y}),&D^2_{yx}\phi(\hat{x},\hat{y})\\
  D^2_{xy}\phi(\hat{x},\hat{y}),&D^2_{yy}\phi(\hat{x},\hat{y})
  \end{array}
  \right]
  $$
  and $||M||$ is the norm given by the maximum, in absolute value, of the eigenvalues of the symmetric matrix $M\in S^{2n}.$
  \end{theorem}
  
   \begin{lemma}\label{prelemma1}
    Let $\phi(x,y)=|x-y|^{\alpha}.$ If $x\not=y$ then 
  \begin{equation}
  \begin{split}
 D^2\phi(x,y)= \left[\begin{array}{cc}
 M,&-M\\
  -M,&M
  \end{array}
  \right],
  \end{split}
  \end{equation}
where 
\begin{equation*}
  \begin{split}
  M=L\alpha |x-y|^{\alpha-2}\left((\alpha-2)\frac{x-y}{|x-y|}\otimes\frac{x-y}{|x-y|}+I\right),
   \end{split}
  \end{equation*}
  
   \begin{equation}
  \begin{split}
  (D^2\phi(x,y))^2=2\left[\begin{array}{cc}
 M^2,&-M^2\\
  -M^2,&M^2
  \end{array}
  \right],
  \end{split}
  \end{equation}
 and
 
  \begin{equation}
  \begin{split}
  M^2=\alpha^2L^2 |x-y|^{2(\alpha-2)}\left(\alpha(\alpha-2)\frac{x-y}{|x-y|}\otimes\frac{x-y}{|x-y|}+I\right).
   \end{split}
  \end{equation}
  \end{lemma}
  \begin{proof}
  The proof follows by straightforward calculation.
  \end{proof}
  It is well known, at least since \cite{Ishii_Lions}, that 
  viscosity solutions of the equation
  $$
  F(D^2u(x))=f(x), \quad \Omega,
  $$ 

 $F$ uniformly elliptic, in the usual sense (see \cite{CC}), homogeneous of degree one, are $C^{0,\alpha}$ in every ball $B\subset 4B\subset\subset \Omega,$ whenever $f\in C(\Omega).$
  
  We want to adapt previous result to the case of degenerate elliptic operators that we are dealing with in this paper. Before doing this, we recall in the next subsection this approach.

\subsection{$C^\alpha$ regularity for uniformly elliptic operators without Harnack inequality}

It is possible to prove $C^\alpha$ regularity of viscosity solutions without proving first the Harnack inequality. Indeed it is sufficient to reduce the problem to a ball of radius $1$ for a non-constant function $0<u<1.$
The scheme of the proof, see for example the idea in \cite{LuiroParviainen} or in \cite{Imbert_Silvestre} for a slightly different but equivalent approach, is the following one:

Let $w(x,y)=u(x)-u(y)-L|x-y|^\alpha-2|x-z|^2,$ for every $z\in B_{1/4}$ 
and denote $\phi(x,y)=L|x-y|^\alpha$ so that
$w(x,y)=u(x)-u(y)-\phi(x,y)-2|x-z|^2,$
Let 
$$
\max_{B_1(0)\times B_1(0)}w(x,y)=w(\hat{x},\hat{y}):=\theta.
$$

Assume by contradiction that $\theta>0.$ Then $\hat{x}\not=\hat{y}.$
Thanks to the localization term $2|x-z|^2$, then $(\hat{x},\hat{y})\in B_{1/4}(0).$ 

By the Theorem of the sums, for every $\mu>0,$
there exist $A=A(\mu)$ and $B=B(\mu)$ such that
  $$
  (D_x\phi(\hat{x},\hat{y}),A)\in\overline{J}^{2,+}u(\hat{x}),\quad (-D_y\phi(\hat{x},\hat{y}),B)\in\overline{J}^{2,-}u(\hat{y}),\quad\mbox{and}
  $$
  \begin{equation*}
  \begin{split}
 \left[\begin{array}{cc}
  A,&0\\
  0,&-B
  \end{array}
  \right] \leq D^2\phi(\hat{x},\hat{y})+\frac{1}{\mu}(D^2\phi(\hat{x},\hat{y}))^2.
  \end{split}
  \end{equation*}

In particular this implies
  \begin{equation*}
  \begin{split}
  \left[\begin{array}{cc}
  A,&0\\
  0,&-B
  \end{array}
  \right] \leq  \left[\begin{array}{cc}
 M,&-M\\
  -M,&M
  \end{array}
  \right]+\frac{2}{\mu}\left[\begin{array}{cc}
 M^2,&-M^2\\
  -M^2,&M^2
  \end{array}
  \right],
  \end{split}
  \end{equation*}
  so that for every $\xi\in \mathbb{R}^n$
  $$
  \langle (A-B)\xi,\xi\rangle\leq 0.
  $$
In addition we conclude that for every $\xi\in \mathbb{R}^n$
  $$
  \langle(A-B)\xi,\xi\rangle\leq 2\langle (M+\frac{2}{\mu}M^2)\xi,\xi\rangle.
  $$
  Moreover, taking $\bar{\xi}=\frac{x-y}{|x-y|}$ and choosing $\mu$ in right way we conclude that:

$$
\langle(A-B)\bar{\xi},\bar{\xi}\rangle\leq L\alpha(\alpha-1)|\hat{x}-\hat{y}|^{\alpha-2}<0.
$$
In this way taking $L$ sufficiently large we obtain a contradiction concluding that $\theta\leq 0.$
Indeed
\begin{equation*}
\begin{split}
&-2||f||_{L^\infty}\leq f(\hat{x})-f(\hat{y})\leq F(A+2I)-F(B)\leq \Lambda \mbox{Tr}(A-B)+n\Lambda\\
&\leq  L\alpha(\alpha-1)|\hat{x}-\hat{y}|^{\alpha-2}+n\Lambda\to -\infty,
\end{split}
\end{equation*}
as $L\to +\infty.$

So that by choosing $z=\hat{x}\in B_{1/4}(0)$ we get that for every $x\in B_{1/4}(0):$
$$
u(x)-u(y)\leq L|x-y|^\alpha.
$$

This proof can be, in a sense, partially adapted to our operators. Nevertheless, see for instance even the subelliptic Laplace operator in Heinseberg group, we did not manage to prove that $\theta<0$ following the previous proof.

Nevertheless, in a paper by Ishii, see \cite{Ishii1}, there is a proof that in some sense works for some, possibly degenerate, linear operators. We remind in the subsection below the main result from our point of view contained in \cite{Ishii1}.

\subsection{A result for linear elliptic operators}

In paper \cite{Ishii1} it was proven the following result.

 If
$$
Lu(x)=\mbox{Tr}(H(x)D^2u(x))+\langle b(x),Du(x)\rangle-c(x)u(x),
$$
where $H^T=H\in C^{1,1}(\mathbb{R}^n,\mathbb{R}^{2n}),$  $b,c,f\in C^{0,1}(\mathbb{R}^n),$ and there exist a matrix $\sigma$ and a  positive number $\Lambda>0$ such that
$
H\geq 0,\quad \sigma^T\sigma=H,
$
and
\begin{equation}\label{bounddness}
 A\leq \Lambda.
\end{equation}
Moreover denoting by
$$
\lambda_0=\sup_{x\not=y}\left\{\frac{\mbox{Tr}(\sigma(x)-\sigma(y))^2+\langle (b(x)-b(y)),x-y\rangle}{|x-y|^2}\right\}
$$
and
$$
c_0=\inf_{\mathbb{R}^n}c.
$$
Then, see \cite{Ishii1},  we get the following result.

\begin{theorem}[Ishii]
Let $c_0\geq 0$ and assume that $c,f\in C^{0,1}(\mathbb{R}^n).$ Let $u\in C(\mathbb{R}^n)\cap L^\infty(\mathbb{R}^n)$  be a viscosity solution of $Lu=f$ that is also uniformly continuous in $\mathbb{R}^n.$ If $c_0>\lambda_0$ then
$u\in C^{0,1}(\mathbb{R}^n)$ and
$$
|Du|_{L^{\infty}(\mathbb{R}^n)}\leq \frac{1}{c_0-\lambda_0}\left(|Df|_{L^{\infty}(\mathbb{R}^n)}+|Dc|_{L^{\infty}(\mathbb{R}^n)}|u|_{L^{\infty}(\mathbb{R}^n)}\right).
$$ 
\end{theorem}

\begin{remark}
If $H(x)=I,$  then $\lambda_0\leq L_b,$ where $L_b$ denotes the Lipschitz constant associated with $b.$ Moreover, if $H(x)=P(x),$ and $b=0,$ that is in the case of the Heisenberg group, then $\mbox{Tr}(P(x)D^2u(x))=\Delta_{\mathbb{H}^1}u.$ Nevertheless condition (\ref{bounddness}) it is not satisfied because $P(x)\leq 1+4(x_1^2+x_2^2).$
Anyhow the approach seems useful to get a first result in the direction we desire as we shall prove in the next Section \ref{mainregularityresulttit}.
\end{remark}

We are now in position to give the proof of our main result.

\section{Proof of Theorem \ref{mainregularityresult}}\label{mainregularityresulttit}

Let
$$
\Phi(x,y)=u(x)-u(y)-L|x-y|^\alpha-\delta |x|^2-\epsilon.
$$
We claim that there exists $L_0(c,||u||_{L^\infty},||f||_{L^\infty})$ such that for every  $\epsilon,\delta>0,$ if $L\geq L_0$ then
$$
\sup_{\mathbb{R}^n\times\mathbb{R}^n}\Phi(x,y)\leq 0.
$$ 
Indeed, arguing by contradiction, if there exist $\epsilon_0>0$ and $\delta_0>0$ such that for $\delta\leq\delta_0,$ $\epsilon\leq\epsilon_0$
$$
\sup_{\mathbb{R}^n\times\mathbb{R}^n}\{u(x)-u(y)-L|x-y|^\alpha-\delta |x|^2-\epsilon\}=\theta> 0,
$$
then invoking Therorem of the sums, see Theorem \ref{CraIsLi} in this paper, and  denoting $\phi=L|x-y|^\alpha,$ we get that
there exist $A=A(\mu)$ and $B=B(\mu)$ such that
  $$
  (D_x\phi(\hat{x},\hat{y}), A+2\delta I)\in\overline{J}^{2,+}u(\hat{x}),\quad (-D_y\phi(\hat{x},\hat{y}),B)\in\overline{J}^{2,-}u(\hat{y}),
  $$
  and the following estimate holds:
  \begin{equation*}
  \begin{split}
 \left[\begin{array}{cc}
  A,&0\\
  0,&-B
  \end{array}
  \right]\leq D^2\phi(\hat{x},\hat{y})+\frac{1}{\mu}(D^2\phi(\hat{x},\hat{y}))^2.
  \end{split}
  \end{equation*}
We remark that denoting
  \begin{equation*}
  \begin{split}
  M:=\alpha L|x-y|^{\alpha-2}\left((\alpha-2)\frac{x-y}{|x-y|}\otimes\frac{x-y}{|x-y|}+I\right),
   \end{split}
  \end{equation*}
then, keeping in mind also Lemma \ref{prelemma1},
$$
M\leq \alpha L |x-y|^{\alpha-2}I
$$
  and
   \begin{equation*}
  \begin{split}
  M^2\leq \alpha^2L^2 |x-y|^{2(\alpha-2)}I.
   \end{split}
  \end{equation*}

Thus
 \begin{equation*}
  \begin{split}
&D^2\phi(\hat{x},\hat{y})+\frac{1}{\mu}(D^2\phi(\hat{x},\hat{y}))^2\\
&=\left[\begin{array}{cc}
 M,&-M\\
  -M,&M
  \end{array}
  \right]+\frac{2}{\mu}\left[\begin{array}{cc}
 M^2,&-M^2\\
  -M^2,&M^2
  \end{array}
  \right]\\
  &=\left[\begin{array}{cc}
 I,&-I\\
  -I,&I
  \end{array}
  \right]\left[\begin{array}{cc}
 M,&0\\
  0,&M
  \end{array}
  \right]\\
  &+\frac{2}{\mu}\left[\begin{array}{cc}
 I,&-I\\
  -I,&I
  \end{array}
  \right]\left[\begin{array}{cc}
 M^2,&0\\
  0,&M^2
  \end{array}
  \right]\\
  &\leq \alpha L |x-y|^{\alpha-2}\left( 1+\frac{\alpha L}{\mu} |x-y|^{\alpha-2}\right)\left[\begin{array}{cc}
 I,&-I\\
  -I,&I
  \end{array}
  \right],
 \end{split}
  \end{equation*}
that is
\begin{equation*}
  \begin{split}
  \equiv L\alpha |x-y|^{\alpha-2}\eta \left[\begin{array}{cc}
 I,&-I\\
  -I,&I
  \end{array}
  \right].
\end{split}
  \end{equation*}
  Here $\eta>1$ and  $\eta\to 1$ possibly taking $\mu$ larger and larger.

On the other hand we have to adapt our inequality to the degenerate part of our operator encoded in the coefficients of the matrix in the second order operator.
Thus from 
\begin{equation*}
  \begin{split}
 \left[\begin{array}{cc}
  A,&0\\
  0,&-B
  \end{array}
  \right]\leq  L\alpha |x-y|^{\alpha-2}\eta \left[\begin{array}{cc}
 I,&-I\\
  -I,&I
  \end{array}
  \right],
  \end{split}
  \end{equation*}
it follows that
\begin{equation*}
  \begin{split}
&\mbox{Tr}\left(\left[\sigma(\hat{x}),\sigma(\hat{y})\right] \left[\begin{array}{cc}
  A,&0\\
  0,&-B
  \end{array}
  \right]\left[\begin{array}{l}\sigma(\hat{x})^T\\
  \sigma(\hat{y})^T
  \end{array}\right]\right)\\
  &\leq  L\alpha |x-y|^{\alpha-2}\eta \mbox{Tr}\left(\left[\sigma(\hat{x}),\sigma(\hat{y})\right] \left[\begin{array}{cc}
 I,&-I\\
  -I,&I
  \end{array}
  \right]\left[\begin{array}{l}\sigma(\hat{x})^T\\
  \sigma(\hat{y})^T
  \end{array}\right]\right).
  \end{split}
  \end{equation*}
Performing the computation for both sides of previous inequality we get
  \begin{equation}\label{ineine}
  \begin{split}
  &\mbox{Tr}\left(\sigma(\hat{x})A\sigma(\hat{x})^T-\sigma(\hat{y})B\sigma(\hat{x})^T\right)=\mbox{Tr}\left(\sigma(\hat{x})^T\sigma(\hat{x})A)\right)-\mbox{Tr}\left(\sigma(\hat{y})B\sigma(\hat{x})^T\right)\\
  &\leq  L\alpha |x-y|^{\alpha-2}\eta \mbox{Tr}\left(\sigma(\hat{x})\sigma(\hat{x})^T-\sigma(\hat{x})\sigma(\hat{y})^T-\sigma(\hat{y})\sigma(\hat{x})^T+\sigma(\hat{y})\sigma(\hat{y})^T\right)\\
  &= L\alpha |x-y|^{\alpha-2}\eta \left(\sigma(\hat{x})-\sigma(\hat{y})\right)\left(\sigma(\hat{x})-\sigma(\hat{y})\right)^T\\
  &=L\alpha |x-y|^{\alpha-2}\eta \left(\sigma(\hat{x})-\sigma(\hat{y})\right)^2.
  \end{split}
  \end{equation}
We can now exploit  some information contained in the fact that $u$ is a viscosity solution of the equation.
  Indeed recalling that $\theta>0$ we get
  \begin{equation*}
  \begin{split}
L|\hat{x}-\hat{y}|^\alpha+\delta c_0|x|^2\leq u(\hat{x})-u(\hat{y})
\end{split}
  \end{equation*} 
and
\begin{equation*}
  \begin{split}
Lc_0|\hat{x}-\hat{y}|^\alpha+\delta c_0|x|^2&\leq c_0\left(u(\hat{x})-u(\hat{y})\right)\leq c(\hat{x})\left(u(\hat{x})-u(\hat{y})\right)\\
&= c(\hat{x})u(\hat{x})-c(\hat{y})u(\hat{y})+u(\hat{y})\left(c(\hat{y})-c(\hat{x})\right).
\end{split}
  \end{equation*} 
  By the Theorem of the sums and the definition of viscosity subsolution/supersolution we get
  \begin{equation*}
  \begin{split}
Lc_0|\hat{x}-\hat{y}|^\alpha+\delta c_0|x|^2&\leq c_0\left(u(\hat{x})-u(\hat{y})\right)\leq c(\hat{x})\left(u(\hat{x})-u(\hat{y})\right)\\
&= c(\hat{x})u(\hat{x})-c(\hat{y})u(\hat{y})+u(\hat{y})\left(c(\hat{y})-c(\hat{x})\right)\\
&\leq F(A+2\delta I,\hat{x})-F(B,\hat{y})\\
&+f(\hat{y})-f(\hat{x})+u(\hat{y})\left(c(\hat{y})-c(\hat{x})\right)\\
&=G(\sigma(\hat{x})^T(A+2\delta I)\sigma(\hat{x}))-G(\sigma(\hat{y})^TB\sigma(\hat{y}))\\
&+f(\hat{y})-f(\hat{x})+u(\hat{y})\left(c(\hat{y})-c(\hat{x})\right).
\end{split}
  \end{equation*} 
Now, if $\sigma(\hat{x})(A+2\delta I)\sigma(\hat{x})^T\leq \sigma(\hat{y})B\sigma(\hat{y})^T$ we conclude by the elliptic degenerate property that 
$$
Lc_0|\hat{x}-\hat{y}|^\alpha+\delta c_0 |x|^2\leq f(\hat{y})-f(\hat{x})+u(\hat{y})\left(c(\hat{y})-c(\hat{x})\right)
$$
because $G(\sigma(\hat{x})(A+2\delta I)\sigma(\hat{x})^T)-G(\sigma(\hat{y})B\sigma(\hat{y})^T)\leq 0.$ 

On the contrary, if $$\sigma(\hat{x})(A+2\delta I)\sigma(\hat{x})T>\sigma(\hat{y})B\sigma(\hat{y})^T$$ 

then
 \begin{equation*}
  \begin{split}
&Lc_0|\hat{x}-\hat{y}|^\alpha+\delta c_0|\hat{x}|^2\leq \Lambda\mbox{Tr}\left(\sigma(\hat{x})(A+2\delta I)\sigma(\hat{x})^T- \sigma(\hat{y})B\sigma(\hat{y})^T\right)\\
&+f(\hat{y})-f(\hat{x})+u(\hat{y})\left(c(\hat{y})-c(\hat{x})\right)\\
&=\Lambda\mbox{Tr}\left(\sigma(\hat{x})A\sigma(\hat{x})^T- \sigma(\hat{y})B\sigma(\hat{y})^T\right)+2\Lambda\delta \mbox{Tr}(P(\hat{x}))\\
&+f(\hat{y})-f(\hat{x})+u(\hat{y})\left(c(\hat{y})-c(\hat{x})\right)
\end{split}
  \end{equation*} 

Thus
 \begin{equation}\label{123ma}
  \begin{split}
&Lc_0|\hat{x}-\hat{y}|^\alpha\leq \Lambda\mbox{Tr}\left(\sigma(\hat{x})A\sigma(\hat{x})^T- \sigma(\hat{y})B\sigma(\hat{y})^T\right)\\
&+2\delta \Lambda |\hat{x}|^2{{ { ( \frac{\mbox{Tr}(P(\hat{x}))}{|\hat{x}|^2}-\frac{c_0}{2\Lambda})}}}+f(\hat{y})-f(\hat{x})+u(\hat{y})\left(c(\hat{y})-c(\hat{x})\right)\\
&\leq \Lambda\mbox{Tr}\left(\sigma(\hat{x})A\sigma(\hat{x})^T- \sigma(\hat{y})B\sigma(\hat{y})^T\right)\\
&+2\delta \Lambda |\hat{x}|^2{{ { ( \frac{\mbox{Tr}(P(\hat{x}))}{|\hat{x}|^2}-\frac{c_0}{2\Lambda})}}}+L_f|\hat{y}-\hat{x}|^\beta+L_c|u|_{L^\infty}|\hat{y}-\hat{x}|^{\beta'}
\end{split}
  \end{equation} 

If $|\hat{x}|$ is bounded as $\delta\to 0,$ then $$2\delta \Lambda |\hat{x}|^2{{ { ( \frac{\mbox{Tr}(P(\hat{x}))}{|\hat{x}|^2}-\frac{c_0}{2\Lambda})}}}\to 0.$$
If  $|\hat{x}|$ were unbounded as $\delta\to 0,$ then $2\delta \Lambda |\hat{x}|^2{{ { ( \frac{\mbox{Tr}(P(\hat{x}))}{|\hat{x}|^2}-\frac{c_0}{2\Lambda})}}}\leq 0$ whenever
$$
\limsup_{|x|\to\infty}\frac{\mbox{Tr}(P(x))}{|x|^2}<\frac{c_0}{2\Lambda}.
$$

It remains to evaluate $\mbox{Tr}\left(\sigma(\hat{x})A\sigma(\hat{x})^T- \sigma(\hat{y})B\sigma(\hat{y})^T\right).$ Indeed by recalling inequality (\ref{ineine}) we get
 \begin{equation}\label{new}
  \begin{split}
&\mbox{Tr}\left(\sigma(\hat{x})A\sigma(\hat{x})^T- \sigma(\hat{y})B\sigma(\hat{y})^T\right)\\
&\leq \eta\alpha L|\hat{x}-\hat{y}|^{\alpha-2} \mbox{Tr}(\sigma(\hat{x})-\sigma(\hat{y}))^2\leq \bar{C} \eta\alpha L|\hat{x}-\hat{y}|^{\alpha-2}|\hat{x}-\hat{y}|^{2}\\
&\leq C\alpha L |\hat{x}-\hat{y}|^{\alpha}
\end{split}
  \end{equation} 
  thanks to our hypothesis on $\sigma,$ where  $\bar{C}$ and $C$ are bounded and independent to $\hat{x}$ and $\hat{y}.$ 

Summarizing, we have got that
 \begin{equation*}
  \begin{split}
&Lc_0|\hat{x}-\hat{y}|^\alpha\leq C\alpha\Lambda L |\hat{x}-\hat{y}|^{\alpha}+ L_f|\hat{y}-\hat{x}|^\beta+L_c|u|_{L^\infty}|\hat{y}-\hat{x}|^{\beta'},
\end{split}
  \end{equation*}
  that is
  \begin{equation*}
  \begin{split}
&c_0\leq C\alpha\Lambda+ \frac{L_f}{L}|\hat{y}-\hat{x}|^{\beta-\alpha}+\frac{L_c}{L}|u|_{L^\infty}|\hat{y}-\hat{x}|^{\beta'-\alpha}.
\end{split}
  \end{equation*}  
  So that by taking $L$
 sufficiently large and $\alpha$ sufficiently small ($\alpha<\frac{c_0}{C\Lambda}$), we get  a contradiction.
 Indeed, since 
 \begin{equation*}
  \begin{split}
 L&\leq \frac{1}{c_0-C\Lambda\alpha}\left (L_f|\hat{y}-\hat{x}|^{\beta-\alpha}+L_c|u|_{L^\infty}|\hat{y}-\hat{x}|^{\beta'-\alpha}\right)\\
 &\leq\frac{1}{c_0-C\Lambda\alpha}\left (L_f(\frac{|u|_{L^\infty}}{L})^{\beta-\alpha}+L_c|u|_{L^\infty}(\frac{|u|_{L^\infty}}{L})^{\beta'-\alpha}\right)
 \end{split}
  \end{equation*} 
  so that keeping in mind that $|\hat{x}-\hat{y}|\leq \frac{|u|_{L^\infty}}{L},$ and for instance, if $\beta\leq \beta',$ then
   \begin{equation*}
  \begin{split}
 L^{1+\beta'-\alpha}\leq\frac{1}{c_0-C\Lambda\alpha}
 \left (L^{\beta-\beta'}[f]_{C^{\beta}}|u|_{L^\infty}^{\beta-\alpha}+[c]_{C^{\beta'}}|u|_{L^\infty}|u|_{L^\infty}^{\beta'-\alpha}\right)
 \end{split}
  \end{equation*}

 getting a contradiction fixing

 $$
 L>\left\{\frac{1}{c_0-C\Lambda\alpha}
 \left ([f]_{C^{\beta}}|u|_{L^\infty}^{\beta-\alpha}+[c]_{C^{\beta'}}|u|_{L^\infty}^{1+\beta'-\alpha}\right)\right\}^{\frac{1}{1+\beta'-\alpha}}.
 $$

Thus 
$$u(x)-u(y)\leq L|x-y|^\alpha+\delta |x|^2 +\epsilon$$
and letting $\delta$ and $\epsilon$ go to $0$  we conclude that
$$u(x)-u(y)\leq L|x-y|^\alpha.$$

\section{Conclusions and remarks}\label{concorem}
\subsection{Square root matrices and rectangular matrices}

In case $P$ were a square matrix sufficiently smooth, so that $\sigma=\sqrt{P},$ we have the required regularity of $\sigma$ invoking the result contained in  \cite{LiVi} or \cite{Ishii1} coming from  \cite{SV}. In that case we deduce that $\sqrt{P}$ is Lipschitz continuous whenever $P$ is $C^{1,1}.$ See also \cite{PaXu} for a different type of remark about the properties of the square root matrices.

In case $P$ were obtained as the product of two rectangular matrices, the proof of Lipschitz continuity follows straightforwardly from the regularity of the coefficients of $\sigma$ themselves. In this case we have to assume that $\sigma$ has to be at least Lipschitz continuous. Indeed the case of the Heisenberg group we start from analytic coefficients! See for instance the Heisenberg case traited in the introduction.

\subsection{A little gain}\label{menkIshii}

Recalling the notation used in the proof of Theorem \ref{mainregularityresult}, if we know that $\delta(\Lambda\mbox{Tr}(P(\hat{x})-c_0|\hat{x}|)\to 0$ as $\delta \to 0,$ then we could improve the result symply requiring that
$$
 L>  \frac{[f]_{C^{\beta}}+[c]_{C^{\beta'}}}{c_0-C\Lambda\alpha}.
$$
In the case of the Heisenberg group $\mathbb{H}^1,$ for instance concerning the sublaplacian, we have that the result is true if
$$
4\leq \frac{c_0}{2\Lambda},
$$
because $\mbox{Tr}(P(\hat{x}))=2+4(x_1^2+x_2^2).$ 

\subsection{The Carnot group case}
More in general, in Carnot groups, it results, in the nontrivial case,  that $\sigma(x)=\sigma(x')$
where $x'$ denotes the variables  that do not contain the ones that are identified  with the last stratum of the Lie algebra of the group, see for instance Remark 1.4.4, Remark 1.4.5, Remark 1.4.6 in  \cite{BLU}. Thus:
$$
|\sigma(x)-\sigma(y)|\leq C|x'-y'|.
$$
 As a consequence, recalling the inequality (\ref{123ma}) in the proof of Theorem \ref{mainregularityresult}, or the quantity (\ref{123ip}) entering in the statement of the Theorem \ref{mainregularityresult}, we remark that:
 \begin{equation*}
  \begin{split}
 &2\delta \Lambda |\hat{x}|^2{{ { ( \frac{\mbox{Tr}(P(\hat{x}))}{|\hat{x}|^2}-\frac{c_0}{2\Lambda})}}}=2\delta \Lambda \left( \mbox{Tr}(P(\hat{x}))-\frac{c_0}{2\Lambda}|\hat{x}|^2\right)\\
 &=2\delta \Lambda \left( \mbox{Tr}(P(\hat{x}'))-\frac{c_0}{2\Lambda}|\hat{x}|^2\right)=2\delta \Lambda \left( \mbox{Tr}(\sigma(\hat{x}')\sigma(\hat{x}')^T)-\frac{c_0}{2\Lambda}|\hat{x}|^2\right)\\
 &\leq 2\delta \Lambda \left( (c+\phi(|\hat{x}'|))-\frac{c_0}{2\Lambda}|\hat{x}|^2\right)\\
 &= 2\delta \Lambda c+2\Lambda\delta |\hat{x}|^{2-\epsilon}\left(\frac{\phi(|\hat{x}'|))}{|\hat{x}|^{2-\epsilon}}-\frac{c_0}{2\Lambda}|\hat{x}|^\epsilon\right),
 \end{split}
  \end{equation*}  
for a suitable positive number $\epsilon.$ 

Here $\phi$ is a polynomial function depending only on $|x'|$ whose degree depends on the step of the group. In general, if the step of the group is $p,$ then  the degree is less or equal $2(p-1).$  

In this case, if $\phi$ does not grow up too much the result is true without restriction on the
size of $\frac{c_0}{2\Lambda}.$ For example, if $\phi(r)\sim r^{1+\nu}$ as $r\to \infty$ for some $\nu\in [0,1).$

\subsection{Simple examples}
It is easy to construct some examples. In very low dimension, $n=2,$ we are considering: 
 \begin{equation*}
\sigma=\left[
\begin{array}{llc}
1,&0\\
\end{array}
\right].
\end{equation*}
Then 
 \begin{equation*}
\sigma^T\sigma=\left[
\begin{array}{ll}
1,&0\\
0,&0
\end{array}
\right],
\end{equation*}
so that for every operator like
$$
F(D^2u(x),x):=G(\frac{\partial^2u(x,y)}{\partial x^2}),
$$
where $G:\mathbb{R}\to\mathbb{R}$ is monotone increasing and vanishing at $0,$ $c, f$ Lipschitz continuous $\inf_{\mathbb{R}^2}c=c_0>0,$ we deduce from Theorem \ref{mainregularityresult} that bounded uniformly continuous functions satisfying
$$
F(D^2u(x,y),x,y)-cu=f, \quad \mathbb{R}^2
$$ 
in a viscosity sense are Lipschitz continuous in $\mathbb{R}^2.$

Let 
 \begin{equation*}
\sigma=\left[
\begin{array}{llc}
\frac{x}{1+x^2},&0\\
\end{array}
\right].
\end{equation*}
Then 
 \begin{equation*}
\sigma^T\sigma=\left[
\begin{array}{ll}
\frac{x^2}{(1+x^2)^2},&0\\
0,&0
\end{array}
\right].
\end{equation*}
so that for every operator like
$$
F(D^2u(x),x):=G(\frac{x^2}{(1+x^2)^2}\frac{\partial^2u(x,y)}{\partial x^2}),
$$
where $G:\mathbb{R}\to\mathbb{R}$ is uniformly elliptic, $c, f$ are Lipschitz continuous, $\inf_{\mathbb{R}^2}c=c_0>0 ,$we deduce from Theorem \ref{mainregularityresult} that bounded uniformly continuous functions satisfying
$$
F(D^2u(x,y),x,y)-cu=f, \quad \mathbb{R}^2
$$ 
in a viscosity sense are Lipschitz continuous in $\mathbb{R}^2.$ Other examples can be easily constructed for degenerate structures without group structure. A first embrional approach in this direction can be found in \cite{Sachs} for a Grushin operator.

\subsection{Limits to this approach}
We are not able to improve our result assuming lower regularity on the coefficients. Indeed
if in (\ref{new}) we assume that $|\sigma(x)-\sigma(y)|\leq C|x-y|^\gamma,$ $\gamma\in (0,1]$ then we conclude that
 \begin{equation}\label{new2}
  \begin{split}
&\mbox{Tr}\left(\sigma(\hat{x})A\sigma(\hat{x})^T- \sigma(\hat{y})B\sigma(\hat{y})^T\right)\\
&\leq \eta\alpha L|\hat{x}-\hat{y}|^{\alpha-2} \mbox{Tr}(\sigma(\hat{x})-\sigma(\hat{y}))^2\leq \bar{C} \eta\alpha L|\hat{x}-\hat{y}|^{\alpha-2}|\hat{x}-\hat{y}|^{2\gamma}\\
&\leq C\alpha L |\hat{x}-\hat{y}|^{\alpha-2+2\gamma}
\end{split}
  \end{equation} 
but in order to get a contradiction we need to ask also that $\alpha-2+2\gamma\geq \alpha$ and this happens only if $\gamma\geq 1.$

\subsection{Conclusions}

It is possible to deduce the H\"older regularity of viscosity solutions without knowing the Harnack inequality, under the hypotheses of  Theorem \ref{mainregularityresult}, even for degenerate nonlinear operators. 
Concerning the remark discussed in Subsection \ref{menkIshii}, we can not deduce  that for linear operators like the sublaplacian in the Heisenberg group the result \cite{Ishii1} applies, see also \cite{F} and \cite{FeVe}, since $P(\hat{x})$ might behave like $|\hat{x}|^2$ and $\delta |\hat{x}|^2$ is only bounded by $2|u|_{L^\infty}.$
As a consequence, our result seems new, even in the linear case. It is worth to say, even it is well known in literature, that considering operators in divergence form, by recalling  H\"ormander approach, see \cite{Hormander}, it is possible to prove, as it is well known, much more significative  regularity results.

\end{document}